\documentclass{amsart}
\usepackage{amsmath,amssymb}
\usepackage[all]{xy}

\theoremstyle{plain}

\newtheorem{introconj}{Conjecture}
\newtheorem{theorem}{Theorem}[section]
\newtheorem{proposition}[theorem]{Proposition}

\newtheorem{question}[theorem]{Question}
\newtheorem*{proposition*}{Proposition}

\theoremstyle{definition}

\newtheorem{example}[theorem]{Example}
\newtheorem{examples}[theorem]{Examples}

\theoremstyle{remark}

\newcommand{\secref}[1]{Section~\ref{#1}}
\newcommand{\thmref}[1]{Theorem~\ref{#1}}
\newcommand{\propref}[1]{Proposition~\ref{#1}}

\newcommand{\conjref}[1]{Conjecture~\ref{#1}}

\def\Fib{{\mathrm{Fib}}}

\def\Q{{\mathbb Q}}
\def\C{{\mathbb C}}

\def\Der{\mathrm{Der}}
\def\Hom{\mathrm{Hom}}

\def\cat{\mathsf{cat}}
\def\rank{\mathrm{rank}}

\def\aut{\mathrm{aut}_1}

\def\B{B\mathrm{aut}_1}
\def\secat{\mathsf{secat}}

\begin{document}

\title[Rational Sectional Category of  the Universal Fibration]
{The Rational Sectional Category of Certain Universal Fibrations}

\author{Gregory  Lupton}

\address{Department of Mathematics,
           Cleveland State University,
           Cleveland OH 44115}

\email{G.Lupton@csuohio.edu}

\author{Samuel Bruce Smith}

\address{Department of Mathematics,
   Saint Joseph's University,
   Philadelphia, PA 19131}

\email{smith@sju.edu}

\date{\today}



\begin{abstract}
We  prove that the  sectional category of the universal fibration 
with fibre $X$,  for $X$ any space that satisfies a well-known conjecture of Halperin,  equals one after rationalization.  
\end{abstract}
 
\thanks{This work was partially supported by a grant from the Simons Foundation (\#209575 to Gregory Lupton).}

\maketitle

\section{Main Result} 
We begin with a concise r{\'e}sum{\'e} of the ingredients, then a statement, of our main result.    Notations used  and assertions made here are described in greater detail below.  \emph{Sectional category} ($\secat$) is a numerical invariant of a fibration that extends the notion of \emph{LS category} ($\cat$) of a space to fibrations.  We normalize these invariants: $\secat(p) = 0$ when fibration  $p \colon E \to B$ has a section.   Sectional category plays a role in several interesting applications (e.g. see \cite[\S 9.3]{CLOT} and \cite{Far03}).  

Fibrations  with fibre a fixed space $X$ are classified by a universal fibration  with fibre $X$ \cite{St,  Dold, May}. For fibrations of simply connected spaces, this universal fibration may be identified, up to homotopy, as the map on Dold-Lashof classifying spaces $u_X  \colon\B^*(X) \to \B (X)$  that is induced by  an inclusion $ \aut^*(X) \hookrightarrow \aut(X)$  of connected  monoids  of self-equivalences \cite{D-L59, Got2}.   

 A fibration $p \colon E \to B$  with $B$ simply connected admits a rationalization $p_\Q \colon E_\Q \to B_\Q$ that, homotopically, represents a simplification of $p$.  Then the \emph{rational sectional category} ($\secat_0$) of fibration $p$ is defined by setting  $\secat_0(p) := \secat(p_\Q).$   
 
Now we have general inequalities $\secat_0(p) \leq \secat_0(u_X) \leq \cat_0(\B(X))$,  with $\cat_0(\B(X))$ the \emph{rational LS category} of the classifying space, which is often infinite (see \cite{Gat95, Gat96}).  So it  is natural to ask whether $\secat_0(u_X)$ may be finite.  

Our main result is the following:

\begin{theorem} \label{main} 
Let $X$ be any $F_0$-space that satisfies \conjref{conj:Halperin} below.     Then $\secat_0(u_X) = 1$,
where  $u_X \colon \B^*(X) \to \B(X)$ denotes the universal fibration with fibre $X$.   
\end{theorem}
 
An $F_0$-space is any (non-trivial) space with $H^*(X; \Q)$ and $\pi_*(X) \otimes \Q$ both finite-dimensional and $H^{odd}(X; \Q) = 0$.  Examples of $F_0$-spaces for which \conjref{conj:Halperin} is satisfied include even-dimensional spheres, complex projective spaces,  homogeneous  spaces $G/H$ with $\rank\, G = \rank\,  H$, and finite products of any of these spaces.

\section{Introduction}
   
We continue with a fuller description of the ingredients just indicated.  
We assume all spaces are simply connected CW complexes of finite type.  Let $p\colon E \to B$ be a fibration.
For $n \geq 1$,  set $\secat(p) = n-1$ if $n$ is the minimal number of open sets $U_i$ in a cover of $B$ such that $p$ admits a section over each $U_i$.  
Set $\cat(B) = n-1$ if $n$ is the minimal number of open sets $U_i$ in a cover of $B$ such that each inclusion $U_i \to B$ is nulhomotopic.   
The inequalities  $\secat(p) \leq \cat(B)$ and $\secat(f^*(p)) \leq \secat(p),$ for $f^*(p)$ the pull-back of $p$ by a map $f \colon B' \to B,$ are both proved directly (see, e.g. 
Proposition 9.14 and Exercise 9.3 of \cite{CLOT}, which reference also contains many facts concerning $\secat$ and $\cat$).   Generally speaking, both $\cat$ and $\secat$ are delicate invariants, and difficult to compute.  It is quite surprising, therefore,  that we are able to obtain a global result such as  \thmref{main}.  This is especially so considering that the universal fibration is a rich construction that involves large and complex spaces whose general structure is not well understood.

Because $u_X  \colon\B^*(X) \to \B (X)$ is \emph{universal},  $\secat(u_X)$, when finite,  is an upper bound for the sectional category of every fibration with fibre $X$. 
From properties of rationalization, we have $\secat_0(p) \leq \secat(p)$.  This inequality follows using an alternate characterization of $\secat$ that we indicate at the start of the next section---involving the so-called fibrewise join construction.          
We also have $\secat_0(p) := \secat(p_\Q) \leq \cat(B_\Q)$, with the inequality following as before, applied to the rationalized fibration $p_\Q$.    For simply connected $B$, we have 
$\cat(B_\Q) = \cat_0(B)$, where $\cat_0$ denotes the rational LS category.     By the naturality of rationalization with respect to pull-backs, we obtain   $\secat_0(p) \leq \secat_0(u_X)$ for any fibration $p \colon E \to B$ with fibre $X$.

Stanley has given a complete calculation of the rational sectional category  of  spherical fibrations \cite{Stanley}.  His results for the even-dimensional sphere imply (in our normalized notation) that $\secat_0(u_{S^{2n}}) = 1$.   Here, we  extend  Stanley's result  from $S^{2n}$ to any $F_0$-space that satisfies Halperin's Conjecture in rational homotopy theory.   We  discuss this class of spaces now.  
  
As above, an  {\em $F_0$-space} $X$  is  an {\em elliptic space}---$H^*(X; \Q)$ and $\pi_*(X) \otimes \Q$  are both  finite-dimensional---with  evenly-graded rational cohomology:  $H^{odd}(X; \Q) = 0.$    Halperin has conjectured the following  generalization of classical results on the rational cohomology of homogeneous spaces. 

 \begin{introconj}[Halperin]\label{conj:Halperin}
Let $X$ be an $F_0$-space and  $p \colon E \to B$ any fibration of simply connected spaces with fibre $X$. Then the rational Serre spectral sequence for $p$ collapses at the $E_2$-term.
 \end{introconj}

The conjecture  is equivalent to the assertion that $\Der_{> 0}(H^*(X; \Q)) = 0$   for $X$ an $F_0$-space   where $\Der_{>0}(A)$ is the   graded Lie algebra of  degree-lowering derivations  of the algebra $A$ \cite{Thomas, Meier}.  The conjecture follows easily from this version for $X = S^{2n}, \C P^n$ and, more generally, for any  space with rational cohomology a truncated polynomial algebra.  Meier  proved  the conjecture  for flag-manifolds $G/T$ with $T$ a maximal torus, and other homogeneous spaces \cite[Th.B]{Meier}.      Shiga and Tezuka  extended Meier's result to the general case of  homogeneous spaces $G/H$ of equal rank pairs \cite{St}.   Halperin's Conjecture has also been confirmed for the cases in which $H^*(X; \Q)$ has $3$ or fewer generators \cite{Lupton}.  Markl has shown that the class of spaces for which  \conjref{conj:Halperin} holds is closed under fibrations, not just products \cite{Markl}.
Our main result, Theorem \ref{main},  applies in all these cases.   We refer the reader to  \cite[p.516]{F-H-T01} for a discussion and other references.

Meier made the following connection between Halperin's Conjecture and the rational homotopy of the universal fibration \cite[Th.A]{Meier}\begin{theorem}[Meier]  \label{meier}Let $X$ be an $F_0$-space.  Then $X$ satisfies Halperin's Conjecture if and only if $\B(X)$ is rationally equivalent to a product of even-dimensional Eilenberg-Mac~Lane Spaces.  $\qed$
 \end{theorem}
Meier's Theorem implies  that   $\cat_0(\B(X)) = \infty$ for $X$ an $F_0$-space that satisfies Halperin's Conjecture.  In fact, this is the case for any elliptic space $X$  \cite{Gat95}.  
 We will use  \thmref{meier} to deduce  \thmref{main} in \secref{sec2} as a consequence of \propref{prop: odd spheres fibre} a technical result  concerning sections of rational fibrations.

 \section{Rational Sectional Category of a Fibrewise Join} \label{sec2}

We recall a special case of a characterization of $\secat$ in terms of the fibrewise join construction.   
 Given  $p \colon E \to B$  with fibre $X$, the fibrewise join $p\ast p$  is  a fibre sequence: 
 $\xymatrix{ X \ast X \ar[r] & E \ast E \ar[r]^-{p \ast p} & B.}$   As is well-known, we may identify the (ordinary) join as  $X \ast X = \Sigma(X \wedge X).$
  The following is a  special case of  a result of Svarc (see \cite[Prop.8.1]{James}): 
\begin{proposition}\label{join}
Let $p \colon E \to B$ be a fibration.  Then $\secat(p) \leq 1$ if and only if   $p\ast p$ has a section. $\qed$
\end{proposition}
Our main result is that the fibrewise join of the universal fibration
$$ \xymatrix{X \ast X \ar[r] & \B^*(X) \ast \B^*(X) \ar[rr]^{\, \, \, \, \, \, \, \, \, \, \, \, u_X \ast u_X} && \B(X)} $$
has   a section after rationalization when $X$ is  an $F_0$-space that satisfies Halperin's Conjecture. 
We make use of the  correspondence between fibre sequences of rational spaces and relative Sullivan models.   
Although we will recall some basic facts about minimal models, our proofs assume a working familiarity with them.
Our reference for rational homotopy is \cite{F-H-T01}. 
 
Let  $ X \to  E \stackrel{p}{\to}  B$  be a fibre sequence.  The {\em relative Sullivan model} for $p$  is a sequence
$$
\xymatrix{  \land W, d_B \ar[r]^{\! \! \! \! \! \!  J}  &  \land W \otimes \land V, D
\ar[r]  & \land V, d_X}$$
of  DG algebras, with $ \land W, d_B$ and  $\land V, d_X$ the Sullivan minimal models for $B$ and $X$, respectively \cite[Prop.15.5]{F-H-T01}.  The differential $D$ satisfies
$$D(w) = d_B(w) \hbox{\, for \, }  w \in W \hbox{\,  and  \, } D(v) -d_X(v) \in \land^{+} W \otimes \land V  \hbox{\, for \, }  v \in V.$$  
The    inclusion $J$ is a model for $p$.  Applying spatial realization, we obtain that  $p_\Q \colon E_\Q \to B_\Q$   admits  a section if and only if $J$ has a left-inverse $S$. That is, $p_\Q$ admits a section if and only if there is a DG algebra map $S \colon  \land W \otimes \land V, D \to \land W, d_B$  with  $S \circ J = \mathrm{id}_{\land W}$.  We prove:  

\begin{proposition}\label{prop: odd spheres fibre}
Let $ X \to  E \stackrel{p}{\to}  B$  be a fibre sequence.  Suppose  
 \begin{itemize} 
\item[(1)] $B$  has the rational homotopy type of  a product of even-dimensional Eilenberg-Mac Lane spaces, 
and 
\item[(2)] $X$ has the rational homotopy type of  a wedge of at least two odd-dimensional spheres.  
\end{itemize}
 
Then the rationalization $p_\Q \colon E_\Q \to B_\Q$ of $p$  admits a section.  
\end{proposition} 

\begin{proof} In the following, we make use of the identification $V^n \cong \Hom\big( \pi_n(X), \Q\big)$, where $(\land V, d)$ is the minimal model of $X$, and also the identification between Samelson products in the rational homotopy Lie algebra  $\pi_*(\Omega X) \otimes \Q$ and the quadratic part of  the differential in  $(\land V, d)$, for any space $X$.  See \cite[Th.15.11, Th.21.6]{F-H-T01} for details.
Hypothesis (1) implies that the minimal model $\land W, d_B$ for $B$ has trivial differential, $d_B = 0,$
with $W^{\mathrm{odd}} = 0.$         We deduce two consequences of  (2) for the minimal model of the fibre,  $\land  V, d_X$.  First we see  that $V^{\mathrm{even}} = 0$ (a wedge of odd-dimensional spheres has no non-zero rational homotopy in even degrees).  Write $V = \langle v_1, v_2, \ldots \rangle$ with $|v_i| \leq |v_j|$ whenever $i < j$ and each $|v_i|$ odd.  Then, for degree reasons, we have $d_X(v_1) = d_X(v_2) = 0$.     The rational homotopy Lie algebra $\pi_*(\Omega X) \otimes \Q$ is free as graded Lie algebra.  Translated to Sullivan models, we deduce, in particular, that there is some    $v_k$ for  $k \geq 3$, such that $d_F(v_k) = v_1  v_2$.

The relative Sullivan model for $X \to E \to B$ is of  the form:  
$$\land W, 0  \to \land W \otimes \land V, D \to \land V, d_X.$$
Given a subspace $V' \subseteq V$, let  $ I(V' ) = \land W \otimes \land ^{+}V'$ denote  the ideal generated by $V'$ in $\land W \otimes \land V.$ Our goal is to prove that $I(V)$  is a $D$-stable ideal of $\land W \otimes \land V.$  We may then define a DG algebra map $S \colon \land W \otimes \land V, D \to \land W, 0$   by $S(w) = w$ and $S(v) = 0$ to obtain the desired section.

We use induction on $i$ to show that, for any $v_i$, we have $D(v_i) \in I(V)$. 
First, we show that $D(v_1) = 0$.  For suppose that $D(v_1) = P$, for some polynomial $P \in \land W$.   Since $d_X(v_1) = d_X(v_2) = 0$,   we see that  $D(v_2) = P_2$ for some $P_2 \in \land W$.  Then, with $v_k$ chosen as above,  we have $D(v_k) = v_1 v_2 +P_k$, for some $P_k \in \land W$.  Notice that there cannot be a term from $\land ^{+}W \otimes \land^{+} V$ in $D(v_k)$, since $V$ is oddly graded and $D(v_k)$ is of even degree, and also because $v_1v_2$ is of minimal degree in $\land^2V$.  Furthermore,  since   $v_1 v_2$ is  the only term from $\land^2 V$ appearing in $d(v_k)$, $v_1v_2$ is the only such term appearing in $D(v_k)$.   Since $D^2(v_k) = 0$, we have
$$0 = D^2(v_k) = D(v_1  v_2 + P_k) = P v_2 - v_1 P_2,$$
and it follows that $P = 0$ (also that $P_2 = 0$). 
 
Now suppose that we have $D(v_i) \in I(V)$ for all $i < t$ and for some $t \geq 2$.  Then $I(v_1, \dots, v_{t-1})$ is a $D$-stable ideal of $\land W  \otimes \land V$.  We may take the quotient  by this ideal yielding the graded algebra 
$$\land W  \otimes \land V_{i \geq t}  =  \left(\land W \otimes \land V \right) / I(v_1, \dots, v_{t-1}).$$ 
Since $I(v_1, \dots, v_{t-1})$ is $D$-stable, $D$ induces a differential $\overline{D}$ on the quotient. 
 Let  $\pi \colon \land W  \otimes \land V, D  \to \land W \otimes \land V_{i \geq t}, \overline{D}$ denote the projection which is       a map of DG algebras.    
 
 Next  observe that $\land(v_1, \ldots, v_{t-1})$ is a $d_X$-stable  sub-algebra of $\land V, d_X$.  Write $\overline{d_X}$ for the induced differential and  $\pi' \colon \land V, d_X \to  \land V_{i \geq t}, \overline{d_X}$ for the projection. 
We claim that  $\land V_{i \geq t}, \overline{d_X}$ is the minimal model for a wedge of odd-dimensional spheres.   For observe that  $\pi'$   is surjective on generators.  Applying  the Mapping Theorem \cite[Th.29.5]{F-H-T01}, we deduce that  $$\cat(\land V_{i \geq t}, \overline{d_X}) \leq \cat(\land V, d_X) = \cat_0(X) = 1.$$
Our claim  follows (recall that $V$, hence $V_{i \geq t}$, is concentrated in odd degrees). 
 
Now consider the   commutative diagram of relative Sullivan models:  
$$\xymatrix{ 
\land W, 0 \ar[r] \ar@{=}[d] & \land W \otimes \land V, D \ar[r] \ar[d]^{\pi} & \land V, d_X \ar[d]^{\pi'} \\
\land W, 0 \ar[r]  & \land W \otimes \land V_{i\geq t}, \overline{D} \ar[r]  & \land V_{i\geq t}, \overline{d_X} . 
}$$
The bottom sequence of this diagram corresponds to a fibre sequence with the original base, say   $\overline{X}_\Q \to \overline{E}_\Q \to B_\Q$.   We have argued above that the fibre  $\overline{X}_\Q$,  with minimal model $\land V_{i\geq t}, \overline{d_X}$,   is   a wedge of odd-dimensional spheres. 
We can now apply  the first part of the argument above, to deduce that $\overline{D}{v_t} = 0$.  But this implies that $D(v_t) \in I(v_1, \ldots, v_{t-1}) \subseteq I(V)$.  
 By induction, we conclude that $I(V)$ is $D$-stable.
\end{proof}

\begin{examples} We allustrate that \propref{prop: odd spheres fibre} is a sharp result, at least if we consider rational fibrations with fibre a wedge of spheres and base a product of Eilenberg-Mac Lane spaces.

\noindent\textbf{(a)}  We must have only even-dimensional Eilenberg-Mac Lane spaces in the base.  For consider   $\Omega S^3 \ast \Omega S^3 \to   S^3 \vee S^3 \stackrel{p}{\to} S^3 \times S^3$ obtained by converting the inclusion to a fibration.  Here, the fibre $\Omega S^3 \ast \Omega S^3 $ is rationally  a wedge of odd-dimensional spheres.  We can see that $p_\Q$ does not admit a section, however, since  $p^* \colon H^*(S^3 \times S^3; \Q) \to H^*(S^3 \vee S^3; \Q)$ is not injective.   

\noindent\textbf{(b)}  We must have a wedge of at least two odd-dimensional  spheres in the fibre.  For example,   the path-loop fibration $K(\Q, 2k-1) \to  PK(\Q, 2k)  \to K(\Q, 2k)$  does not have a section. In fact,  since the total space here is contractible we have   $\secat_0(p) = \cat(K(\Q, 2k))  = +\infty$.
  Recall that we have $S^{2k-1} \simeq_\Q K(\Q, 2k-1).$ 

\noindent\textbf{(c)} 
We  cannot allow even-dimensional spheres in the  fibre.    For suppose $E'$ is a wedge of spheres  and set $E = S^{2k} \vee E'$.  By pinching off $E'$, then including the sphere as the bottom cell of an  Eilenberg-Mac Lane space, we obtain a fibre sequence $ X \to  E \stackrel{p}{\to} K(\Q, 2k)$  
with fibre $X$ a wedge of spheres (of both odd and even dimensions).  This fibration cannot admit a section, as $p^*$ is  not injective in cohomology.  
 \end{examples}

We next observe that the universal fibration $u_X$ for $X$ a nontrivial  elliptic space does not admit a section.  To prove this, we use the  {\em Gottlieb group} $G_*(X) \subseteq \pi_*(X)$ \cite{Got1}.  Recall  that $G_*(X) = \mathrm{Image}\{ \omega_\sharp \colon \pi_*(\aut(X) \to \pi_*(X) \} $ where $\omega \colon \aut(X) \to X$ is the evaluation map.

\begin{proposition}\label{elliptic}  Let $X$ be a nontrivial elliptic space, such as any $F_0$-space.  Then $\secat_0(u_X) \geq 1.$
\end{proposition}

\begin{proof}  First note that we have $G_*(X_\Q) \not= 0$ for any elliptic space $X$.  This is because, in the minimal model $(\land V, d)$ for $X$,  there is an (odd) integer $n > 0$ such that $V^n \neq 0$ and $V^m = 0$ for $m > n$.   Then it follows from the identification of the Gottlieb group in minimal model terms, discussed in \cite[\S 29d]{F-H-T01}, that we have $G_{n}(X_\Q) \cong V^{n} \not= 0$.     Next, by \cite[\S 4]{Got1},  $G_*(X)$ corresponds to the image of 
$\partial \colon \pi_{*+1}(\B(X)) \to \pi_*(X)$, the linking homomorphism in the long exact sequence  of the universal fibration.   Thus $G_*(X_\Q) \neq  0$ implies that $(u_X)_\Q$ does not induce a surjection on rational homotopy groups and so cannot admit a section.  \end{proof}

We apply the preceding  to prove our main result:

\begin{proof}[Proof of \thmref{main}]
Let $X$ be a nontrivial  $F_0$-space.   By Proposition \ref{elliptic}, we have $\secat_0(u_X) \geq 1.$ We   prove $\secat_0(u_X) \leq 1.$   

The fibrewise join of the   universal fibration, $u_X \ast u_X,$ has  base $\B(X)$ and fibre $X \ast X = \Sigma(X \wedge X)$.  By Theorem \ref{meier}, if $X$ satisfies Halperin's Conjecture then $\B(X)$ is rationally a product of even-dimensional Eilenberg-Mac Lane spaces.  Regarding the fibre, note that,    since  $H^*(X;\Q)$ is evenly graded,   $H^*(\Sigma(X \wedge X);\Q)$ is oddly graded. By \cite[Th.1.5]{HS},  $X \ast X$ has the rational homotopy type of  a wedge of odd-dimensional spheres.  If $\widetilde{H}^*(X;\Q)$ has dimension at least $2$, then 
$X \ast X$ is rationally a wedge of at least two odd-dimensional spheres.  Applying  \propref{prop: odd spheres fibre}, we conclude the rationalization of $u_X \ast u_X$ has a section and  we conlcude $\secat_0(u_X) \leq 1$ 
by Proposition \ref{join}. 

When   $\widetilde{H}^*(X;\Q)$ has dimension $1,$ then $X \simeq_\Q S^{2n}$   and we can invoke Stanley's result \cite[Lem.3.1]{Stanley}.  Alternately, we may observe that the   fibrewise join $u_{S^{2n}} \ast u_{S^{2n}}$ has fibre $S^{2n} \ast S^{2n} \simeq_\Q K(\Q, 4n -1)$ and base $\B(S^{2n})\simeq_\Q K(\Q, 4n).$  It follows easily from degree considerations that $(u_{S^{2n}} \ast u_{S^{2n}})_\Q$ is fibre-homotopically trivial and so, in particular, has a section. 
\end{proof}

Theorem \ref{main} reduces the computation of $\secat_0$ for  fibrations with fibre $X$  satisfying Halperin's Conjecture  to    the question of the existence of a section. Stanley expressed the obstruction to a section   in cohomological terms when  $X = S^{2n}$ \cite[Th.3.3]{Stanley}.  We follow his approach to obtain the following example. 

 \begin{example} Write $\Fib_{\C P^m}(\C P^n)$ for the set of   fibrations   $\C P^m \to E ~\stackrel{p}{\longrightarrow} \C P^n$ 
modulo   rational fibre-homotopy equivalence.   We assume $m < n.$   The identity  $\B(\C P^m) \simeq_\Q \prod_{k=1}^{m} K(\Q, 2k)$ follows from  \cite[Pro.2.6(iii)]{Meier}.  By universality, 
 $$ \Fib_{\C P^m}(\C P^n) \equiv [\C P^n, \B(\C P^m)_\Q] \equiv \oplus_{i=1}^{m}H^{2k}(\C P^n; \Q)  \equiv    \Q^{m}. $$
 We  associate  the  rational  fibre-homotopy equivalence class  of  $p$  with an explicit $m$-tuple $(a_{m-1}, \ldots, a_0) \in \Q^m$   defined as follows.  The relative Sullivan model  for $p$ is an inclusion $\land(x_2, y_{2n+1}; d) \to \land(x_2, y_{2n+1}) \otimes \land(u_2, v_{2m+1}), D$
 with   $Dx = dx = 0,  Dy = dy = x^{n+1}, Du = 0$.  As for $Dv,$ we have $$Dv = u^{m+1} + a_m u^mx + a_{m-1}u^{m-1}x^2 + \cdots + a_0x^{m+1} $$  for some $a_i \in \Q.$ The basis change $u \mapsto u -\frac{1}{m+1}a_m x$   will depress the polynomial and we may assume $a_m = 0$. 
 
 Given a section  $S \colon \land(x, y) \otimes \land(u, v)  \to   \land (x, y)$   write $S(u) = qx$ for $q \in \Q$.  Comparing  coefficients of $x^{m+1}$ in the equation $S(Dv) = DS(v) = 0$ we see that $z = q$ is a solution to $z^{m+1} + a_{m-1}z^{m-1} + \cdots + a_1z + a_0 = 0.$     The converse follows similarly and   we obtain:  $$\secat_0(p) = \left\{ \begin{array}{ll} 0 & \hbox{\, if } z^{m+1} + a_{m-1}z^{m-1} +  \cdots +a_1z   + a_0  \hbox{\, has a rational root} \\ \\
1 & \hbox{\, otherwise}.
\end{array} \right.$$    
\end{example}

   We conclude with a   question arising from our work.   Meier \cite{Meier} and others have given various equivalent versions of  Halperin's Conjecture.  It would be interesting to have an equivalent version of the conjecture phrased in terms of the sectional category of the universal fibration.   We pose the following: 
\begin{question} Let $X$ be an $F_0$-space.  Does $\secat_0(u_X) = 1$ imply $X$ satisfies  Halperin's Conjecture?  
\end{question}

     
\providecommand{\bysame}{\leavevmode\hbox to3em{\hrulefill}\thinspace}
\providecommand{\MR}{\relax\ifhmode\unskip\space\fi MR }
\providecommand{\MRhref}[2]{%
  \href{http://www.ams.org/mathscinet-getitem?mr=#1}{#2}
}
\providecommand{\href}[2]{#2}

\end{document}